\documentclass{article}  
\usepackage{amsthm}
\usepackage{amssymb}
\newtheorem{theorem}{Theorem}
\newtheorem{defn}{Definition}

\newtheorem{prop}{Proposition}

\newtheorem{lemma}{Lemma}
\begin{document}
\title{Real analytic Bergman spaces}
\author{Mark G. Lawrence}
\maketitle
\begin{abstract}
The theory of CR wedge extension is combined with a study of moment conditions to construct a new class of Bergman-type spaces which are characterized by real analyticity, rather than holomorphicity. The spaces have dense subsets of real analytic functions which contain  entire functions as a proper subset. 
\end{abstract}
In this paper, the idea of a Bergman space is given a new class of examples, with novel properties. The main theorems are extensions of the author's results in \cite{La2}. The distinguishing property of all of the examples is that real analyticity is preserved under closure, due to moment conditions, but without requiring holomorphicity. 

This paper can be considered a combination of two strands of mathematical research which unexpectedly converged. On the one hand, there is the theory of Bergman spaces, now  100 years old; a very rich theory with potent applications in the sciences, and an active area of research in its own right. The other strand is polynomial approximation theory in several complex variables.
With Weierstrass' theorem,  mathematicians began studying the problem of finding the closure of a finitely generated algebra of functions on a compact set. A basic application of the $n$-dimensional version of the Weierstrass approximation theorem is the observation that on $\bf C^n$, the algebra generated by  $z, \overline z$ is dense in the continuous functions on any compact set. The idea that $\overline z$ makes an algebra consist of continuous functions is also seen in the Wermer theorem, which states that if $\overline z$ is added to the disc algebra $A(D)$, considered as an algebra of functions on the circle, then the continuous functions are the closure of the new algebra. For general information about Bergman spaces, see \cite{Zhu}. For information about polynomial approximation theory, see \cite{WerA}
and \cite{StPC}.

Polynomial approximation in several complex variables  has new features, especially when considering approximation on sub manifolds. In that case, as is well known, the  existence of complex tangent directions often guarantees that functions in the polynomial algebra extend to be holomorphic on some defined set, independent of the function. 
These theorems exists for smooth manifolds, but also for piecewise smooth unions of manifolds in $C^n$. The theory of analytic continuation in this setting falls under the heading of "edge of the wedge" theory. See \cite{Ber} for material about polynomial approximation in the CR setting, as well as wedge extension. 

Starting around the year 2000, Agranovsky and Globevnik, \cite{AgGl}, among others, noted  that analytic continuation techniques in several variables could be used to study regularity problems involving moments. 
A notable success of this idea was Tumanov's solution of the "strip problem", \cite{Tu}, which showed that several complex variables techniques could be used to solve problems in one variable where no one-variable technique seemed to suffice. 
The author learned of these techniques while working on an  a proof of the strip problem for $L^p$ functions (most of the theory in this area is for continuous or real analytic functions). 
In his 2015 paper \cite{La2}, he showed the first examples of using CR wedge extension to prove {\it real analyticity} of a function which satisfied some moment conditions, where there were examples satisfying the moment conditions which were not holomorphic. Although the results are expressed in terms of moment conditions, another formulation is to say that for certain entire functions $g(z)$, the closure of the  algebra generated by $z$ and $\overline z g(z)$ consists of real analytic functions with an infinite radius of convergence. This comes back full circle to the very early notion that $\overline z$ is "too much" to put into an algebra and get an interesting example. If there is a twist with the multiplication by $g(z)$, then real analyticity is maintained under closure. In this paper, further restrictions on the function $g$, and new techniques, allow us to  prove bounded point evaluation, from which closure in the $L^p$ space follows easily by normal families.   

Another way of looking at the result on Fr\'{e}chet algebras is to pose the following question. 
Consider algebras  of functions $\cal A$ on $\bf C$, with the compact open topology.
Are there examples  $\cal A$ with ${\cal O}({\bf C})\subsetneqq {\cal A}\subsetneqq C({\bf C})$, closed in the compact open topology, such that every function in $\cal A$ has some guaranteed smoothness more than continuity? As far as the author knows, this is a new question---at least for the case of real analyticity, this question may be new. 
There is some analogy with the idea of Douglas algebras. These are the algebras in $L^{\infty}(S^1)$ which are in between $H^{\infty}$ and $C(S^1)$. A very rich structure was discovered in the setting of Douglas algebras. Our research shows that there are some tractable questions about Fr\'{e}chet algebras between ${\cal O}(\bf C)$ and $C(\bf C)$ whose answers hint at a new theory. See \cite{Garnett} for background material on Douglas algebras.

After some background material, the basic result is proved in section 2 and 3. Various examples and generalizations are considered in later sections. The full potential of the new technique is not completely outlined here. In particular, Bergman spaces of real analytic functions on $\bf C^n$ and on certain Stein spaces, following the methods of \cite{La2} can be constructed in a fairly straightforward fashion.
There are two methods of proof of the estimates here. One of them uses well known integral formulas of Airapetyan and Henkin along with estimates which are  specific to the 1-dimensional extension problem (and do not hold in general).
With these techniques we can show that bounded point evaluation holds in $L^2$, for suitable $g(z)$ and with the Gaussian weight.
Then we will demonstrate a different technique based on a standard type of construction of analytic discs with boundary in the CR wedge.
The advantage of the second method is one gets control over the region of integration corresponding to the point evaluation for a given point. This allows us to prove some precise theorems about order of growth for a natural class of examples.

\section{Basic notions}
Let $X\subseteq \bf C$ be a discrete set. ${\cal A}_X $ is the set of all continuous $f(z), z\in \bf C$ such that for every $a\in X$ and every $r>0$, $f|_{|z-a|=r}$ extends holomorphically to  $|z-a|<r$.  
If the set $X$ contains vertices of arbitrarily large triangles which contain 0 roughly in the center, none of whose angles degenerate near $\infty$, then every function in ${\cal A}(X) $ is real analytic with infinite radius of convergence. The precise statement is in Theorems 1 and 2 of \cite{La2}. The basic ideas will be explained anew, because we modify them for the construction of Bergman spaces.
In case these conditions are met, then we further showed that 
 ${\cal A}_X$ is generated by $z$ and $\overline zg(z)$, where $g(z)$ is an entire function vanishing exactly on $X$, with only simple zeroes. 
 
A {\it Bergman space} means a closed subspace of $L^p$ of a domain (possibly with a weight) such that there is bounded point evaluation at every point in the domain, for functions in the space. This definition was given to the author by F. Haeslinger, who also posed the question of whether one could construct Bergman spaces of real analytic functions, based on the example ${\cal A}_X$. We are very grateful to Professor Haeslinger for suggesting the problem. 

\section{Construction of the algebras; the CR extension with $L^p$ estimates}
 There are always non-holomorphic functions in ${\cal A}_X$, for if $g(z)$ is an entire function whose zero set contains $X$, then $\overline z g(z)\in {\cal A}_X$. 
Given a discrete set $X\subseteq \bf C$, and a point $a\in X$, 
let $M(a)$ denote the associated half $CR$ manifold of extension of functions in ${\cal A}_a$.
$$M_a=\{(z,w): w=\frac{t}{z-a}+\overline a, t\ge |z-a|^2\}.$$ $M_a$ is a union of punctured analytic discs, and every function which is in ${\cal X}_a$ has a CR lift to $M_a$ (details in \cite{La2}).  Given any two points, $a,b\in X$, there is CR wedge extension from $M_a\cup M_b$. For any $z$, the extension is into the region between the rays of $M_a$ and $M_b$; in the exceptional case where $z$ is on the line joining $a$ and $b$, wedge extension does not (directly)  show any analytic continuation. If compact subsets of $M_a\cup M_b$ are taken, then local CR extension occurs. This is how the extensions will be used in this paper.

By a change of variables, we can consider the local model 
$M_1=\{(z_1,z_2): y_1=0, y_2\ge 0\}, M_2=\{(z_1,z_2):y_1\ge 0, y_2=0\}$. Given a continuous CR function on $M_1\cup M_2$,  we consider the analytic extension to the wedge $y_1>0, y_2>0$. The analytic extension can be realized by analytic discs with boundary in $M_1\cap M_2$. There are many choices; here we use dilations and translations of a fixed disc. Let $\Delta={\zeta: |\zeta |\le 1}$. Let $R$ be the upper half of a circle of radius 1, symmetric about the $y$-axis, with the bottom edge on the real axis. 

Let $\phi_1: \Delta\rightarrow \bf C$ be the unique conformal map with the property that $|\phi(|\zeta|<1)|=1$ on exactly on the upper half circle, $Im(\phi)= 0$ exactly on the lower half circle, and $Re(\phi(0))=0$. 
 Denote $\phi(0)=it$.  For any point $z=a+ib$ in the upper half-plane, the map $\phi_{a,b}(\zeta)=\frac{b}{t}(\phi(\zeta)+\frac{at}{b})$ also satisfies $Im(\phi)=0$ exactly on the lower half-circle, maps the upper half-circle to the upper half-plane, and $\phi_{a,b}(0)=a+ib$

Now consider a point $(z,w)=(a+ib,c+id)$ with $b>0, d>0$. The map $\Phi_{a,b,c,d}(\zeta)=(\phi_{a,b}(\zeta),\phi_{c,d}(-\zeta))$ provides an analytic disc such that 
\begin{enumerate}
\item $\Phi(0)=(z,w)$
\item $\Phi(e^{i\theta})\in M_1, \pi \le \theta \le 2\pi$. 
\item $\Phi(e^{i\theta})\in M_2, 0\le \theta \le \pi$.  
\end{enumerate}
We observe that as $z\rightarrow \pm 1$ in  $\overline\Delta$, 
$\phi_{a,b}'$ and $\phi_{c,d}'$ are comparable. Therefore, in Cauchy integrals over $\Phi(|\zeta|=1)$, $|dz|$ and $|dw|$ are comparable. 

For sharper estimates we need an improvement. Let $\epsilon >0$ be given, and $R_1>>\epsilon$.
Given a real horizontal or vertical dilation of the upper half circle, $\widetilde R $, there is a corresponding $t$, and discs $\widetilde\Phi$ can be constructed in the same fashion. We state as a proposition the fact we will use later. The proof is straightforward.
\begin{prop}
Given $\epsilon$, $R_2$, there exists and $R_2, \epsilon<R_2<R_1$, $C<1$ such that if $Z=(a+ib,c+id)$ satisfies $b>\epsilon, d>\epsilon, |a|>\epsilon, |b|>\epsilon, |Z|<R_1$, there exists a dilation of the  upper half disc, $\widetilde R$ such that the associated disc $\widetilde \Phi$ satisfies $|Re(\Phi_1)|>C\epsilon$, $|Re(\Phi_2)|>C\epsilon$ and $|\Phi|<R_2$.
\end{prop}

The point of this proposition is that the real parts of the coordinates determine the leaf in the Levi foliation. For better estimates for the point  0, we will want to stay away from the leaves which contain 0. A disc given by Proposition 1 gives an $L^p$ estimate, where the region of integration in each coordinate plane does not contain 0. 

Denote this analytic disc by $D_{a,b,c,d}$, or $D_p$ for short, where $p=(a+ib, c+id)$. Let $D_{1p}$ and $D_{2p}$ be the projections of $D_p$ onto the $z$ and $w$ planes, respectively;  the images of the components of $\Phi$.
Given a function $f$ which is continuous and CR on $M_1\cup M_2$, we can use the Cauchy formula $f(p)=\int_{\partial D_{1p}}f(z,g(z)) \frac{dz}{2\pi i}$, where $g(z)=\phi_{c,d}\circ (\phi_{a,b})^{-1}(z)$. In order to  estimate $|f(p)|$ we will first consider the integral over the circular part of $\partial D_{1p}$, then we use $w$ and use the same estimates over the circular part of $D_{2p}$. Since $|dz|$ and $|dw|$ are comparable, this gives one estimate. 
The details are in the next few paragraphs.

The $L^p$ estimates are derived from growth estimates for Hardy space functions. If $f\in H^p(S^1)$, then $$|f(z)|\le C \frac{||f||_p}{(1-|z|)^{1/p}}.$$
Suppose $f\in {\cal A}_a$; for convenience, set $a=0$. 
Let $\gamma(t), t\in [0,b]$ be a smooth curve such that $|\gamma(0)|=r_0$, $|\gamma(t)|<r_0-t$, and $r_0-|\gamma(t)|\equiv kt$ for some $k>1$. The radii could be also be decreasing to $r$ with the same growth rate.
Let $f_r$ denote the holomorphic extension of $f|_{|z|=r}$ to the disc $|z|<r$. Then  $|f_r(\gamma(t)|\le\frac{||f_r||_p}{(r-|\gamma(t)|)^{1/p}}$, and

$$\int_{0}^b|f_{r-t}(\gamma(t))|dt\le C\int_0^b\frac{||f_{r-t}||_p}{t^{1/p}}dt.$$

The following lemma gives what is needed to get the point evaluation estimates.

\begin{lemma}
Fix $p>2$. Let $\gamma(t)=(r(t), z(t))$, $0\le t\le b$, $r(t)>0$, $z(t)\in {\bf C}$,  $|z(t)|<r(t)$; $r(t)$ can be assumed to be monotonic in $t$.  Suppose that $\gamma(0)=(r_0, z_0), |z_0|=r_0$, and for some $0<K<1$, $0<Kt<r(t)-|z(t)|<\frac{1}{K}t$. If $f\in {\cal A}_0$, and if $F(\gamma(t))=f_{r(t)}(z(t)$, where $f_r$ is the holomorphic extension of $f$ on the disc of radius $r$, then 
$$\int_0^b |F(\gamma(t))|dt\le C\left(\int_{r_1}^{r_2}|f(z)|^p dA\right)^{1/p}.$$

Here $r_1$ and $r_2$ are the lowest and highest values of $r(t)$. 
\end{lemma}
\begin{proof}
Let $q$ be the conjugate exponent; $\frac{1}{p}+\frac{1}{q}=1$. 
First apply the growth estimate with $p=2$ to get 
$$\int_0^b |F(\gamma(t))|dt\le C \int_0^b\frac{||f_{r(t)}||_2}{t^{\frac{1}{2}}}$$
$$\le C\int_0^b\frac{||f_{r(t)}||_p}{t^{\frac{1}{2}}}\le C\left(\int_0^b\frac{1}{t^\frac{q}{2}}dt\right)^{\frac{1}{q}}\left(\int_0^b\int_{|z|=r(t)}|f(z)|^p|dz|dt\right)^{\frac{1}{p}}\le$$
$$C\left(\int_{r_1\le |z|\le r_2}|f(z)|^pdA\right)^{\frac{1}{p}}.$$

\end{proof}

This  lemma is applied for estimating the contribution of a Cauchy integral on the boundary of some $\Phi_{a,b,c,d}$, on the part of the curve near the edge. By construction, the approach at the edge is transverse, which allows the estimate of the lemma to be applied.
(The only difficulty with estimation is near the edge.)

Next we apply this estimate to get $L^p$ bounds inside the wedge extension of the $F$ associated to $f\in{\cal_A}_{01}$. 
For $z$ with $Im(z)\ne 0$, we have extension on a CR wedge consisting locally of half spaces in $Im(zw)=0$, $Im(zw-z-w)=0$.
Apply the change of coordinates $\Psi(z,w)=(\zeta,\tau)=(zw, zw-z-w)$.
The Jacobian of this map is $z-w$, which is non-zero in a neighborhood of any point $(z,\overline z)$ with $z$ not real. Pick such a $P=(z,\overline z)$ and let $B_1$ be a small ball centered at $P$ so that $\Psi$ is 1-1 on $B_1$ with $|det(\Psi')|$ bounded away from zero. By a translation, and multiplying components of the map by $-1$ if necessary, we may assume that 
$\Phi(z,\overline z)=(0,0)$, and that the image of  our CR wedge is locally in $\{(z',w'):Im(z')\ge 0, Im (w')=0\}\cup \{(z',w'): Im(z')=0, Im(w')\ge 0\}$.
 Let $B_2=B(0,\epsilon)$ be  a ball in the $(z',w')$ coordinates such that $U=B_2\cap \{(z',w'): Im (z')>0, Im(w')>0\}$ is contained in the region of of wedge extension of $\Psi(B_1)\cap(\{z',w'):Im(z')=0, Im(w')\ge 0\}\cup \{(z',w'): Im(z')\ge 0, Im(w')=0\})$.
 Fix a $\delta <<1$. 
 By shrinking $\epsilon$ if necessary,  for any point $(z,w)\in U$, there is a disc $\Phi_{a,b,c,d}$ such that
 $\Phi(|\zeta|=1)\subseteq \frac{1}{2}K.$
Using these discs $\Phi_{a,b,c,d}$ whose images lie in $\Psi(B)$, 
for all points in $U$.
The comparability of $Im(\phi_1)(e^{i\theta})$ and $Re(\phi_2)(e^{i\theta})$ for $\theta$ close to 0 or $\pi$, in the range $[0,\pi]$, and the corresponding statement for $Re(\phi_1)$ and $Im(\phi_2)$ on the lower half of the circle, allow us to apply the lemma, giving the following proposition.

\begin{prop}
Given Given a point $(z,\overline z)$ with $Re z\ne 0$, there  exists a balls $B_1\subseteq B_2$ containing $(z,\overline z)$ such that:
\begin{enumerate}
\item In the local CR extension from $B_2\cap M_0\cap M_1$, every point in $M_{01}\cap B_1$ is contained in a disc $\Phi_{\cdot}$ whose boundary is contained in $B_1$. 
\item For any point $(u,v)\in B_2\cap M_{01}$, $v\ne \overline u$, there exists a constant $C_{u,v}$, independent of $f$, and $R>0$ independent of $f$, such that $$|f(u,v)|\le C_{u,v}\left(\int_{|z|\le R } |f(z)|^p dA\right)^{\frac{1}{p}}.$$
The constants $C_{u,v}$ are bounded on any compact set of $B$ which does not intersect $\{(z,w):(Im(zw)=0\}\cup \{(z,w):Im(zw-z-w)=0\}$
\end{enumerate}
\end{prop}

The holomorphic extension to $B$ is simply what is given by the CR wedge extension. The bounds on the constants $C_{u,v}$ follow from bounds on the maps $\Phi$. 

The estimates blow up near the CR wedge. In the original paper, if the set $X$ associated to ${\cal A}_X$ contained vertices of arbitrary large triangles, one could prove real analyticity.
To get bounded point evaluation, we need some overlap: the analytic continuation from a wedge $M_1\cup M_2$ will cover a third CR half space $M_3$ which is in the region of wedge extension. We need at least 5 local CR wedges to get estimates.
Here is a theorem which gives $L^p$ estimates for point evaluation in a neighborhood of a suitable point. 

\begin{theorem} Let $X\subseteq \bf C$ be a discrete set, and let ${\cal A}_X$ be the associated algebra. 
Let $z_0\in \bf C$ be a point with the properties listed below. Then for some $\epsilon >0$, there exist continuously varying constants $C_z$ for $|z-z_0|<\epsilon$, and an $R>0$, such that for all $f\in {\cal A}_X$, $|f(z)|<C_z\left(\int_{|z-z_0|<R}|f(z)|^pdA\right)^{\frac{1}{p}}$.

There are $5$ points $a_1,\dots, a_5, a_i\in X$ such that
in the fiber over $z$ of $\cup M_i$, every ray of an $M_i$ is contained between two other rays on the side with angle less than $\pi$. 
\end{theorem}

\begin{proof} Take $z$ to be 0. Consider a circle $\gamma =\{(z,w):w=e^{i\alpha z}, |z|=\eta\}$, where $\alpha$ is chosen so that the intersections with the $M_{a_i}$'s are isolated points. Then for small enough $\eta$ (which may require shrinking $\epsilon$ and $\delta$, we can find an $L^p$ estimate on every point of $\gamma$. On points on or close enough to a particular $M_{a_i}$, we can use the wedge extension estimates from two other CR manifolds, as in the proposition.
By Cauchy's formula, this gives the estimate for $z=0$; for nearby points it follows from continuity of all parameters involved.
\end{proof}

It is a straightforward matter to use normal families to show that  functions in $B_{p,X}$ are in ${\cal A}_X$. Since this is an important fact, we state it explicitly
\begin{theorem}
$B_{p,X}\subseteq {\cal A}_X$.
\end{theorem}



\section{Estimates derived from Airapetyan-Henkin formulas}
There is a different way of getting estimates which includes the case $p=2$. The proof uses integral formulas due to Airapetyan and Henkin for wedge extension from a pair of Levi-flats. The exponent is improved, but the weak mean-value property using integration over an annulus cannot be proved with this method. 
Given $M_{a_1}, M_{a_2}$ and some point $p$ not on the line determined by $a_1$ and $a_2$, there is wedge extension on one side. Locally, this is biholomorphically equivalent to the case $M_1=\{(z,w):Re(w)=0, Re(z)>0\}, M_2=\{Re(z)=0, Re(w)>0\}$. 
Intersecting $M_1$ and $M_2$ with a small ball centered at $p$, then for a point $(z,w)$ close enough to 0, with $Re(z)>0, Re(w)>0$, then for any function $f$ which is $CR$ on $M_{a_1}\cup M_{a_2}$ (thus automatically extending to a holomorphic function on a region including $(z,w)$ for $(z,w)$ close enough to 0), 
$$ f(z,w)=\int_{(M_{a_1}\cap M_{a_2}\times[0,1]}K_1(\zeta,\tau,t,z,w)f(\zeta,\tau)dm(z,w)dt $$ $$+\int_{M_{a_1}}K_2(\zeta,\tau,z,w)f(\zeta,\tau)dm_1(\zeta,\tau)+\int_{M_{a_2}}K_3(\zeta,\tau,z,w)f(\zeta,\tau)dm_2(\zeta.\tau).$$
Implicit in these formulas is a cut-off, so the integrations are taken over compact sets. 
The kernel $K_1$ has $(z-\zeta)(w-\tau)$ in the denominator, which is non-vanishing. The kernel $K_2$ has $(w-\tau)$ in the denominator, and $K_3$  has $(z-\zeta)$ in the denominator; in both cases, the fraction is non-bounded and vanishing. We mean here boundedness with the point $(z,w)$ fixed. 

Now we derive bounded point evaluation for $\le p<\infty$. Estimates of the following type are contained in
\cite{La3} and depend on the 1-dimensional extension property. Most likely they do not hold in general. 
\begin{lemma} given $p\in M_{a}\cap (w=Im (z))$, then for sufficiently small balls $B_r(p)$,
$$\int_{B_p(r)\cap M_{a}}|f^2|\le K\int_{B_p(r)\cap (w=\overline z)}|f|^2.$$
\end{lemma}
\begin{proof} Integrals on slices parallel to the totally real plane are bounded by the integral on the totally real plane. This is an application of standard $H^p$ theory on the disc. Details in \cite{La3}
\end{proof}

We conclude that there is bounded point evaluation in $L^p({\bf C}, e^{-|z|^2}), 1\le p<\infty.$

\section{An explicit example, with growth estimates}
Let $X={\bf Z}\cup (\omega {\bf Z})\cup (\omega^2 {\bf Z})$, where $\omega^3=1, \omega \ne 1$. Take $g(z)=\frac{\sin z \sin (\omega z)\sin(\omega^2 z)}{z^2}$. From the order of $g$ we see that 
$$B_X=\{f\in {\cal A}_X: \int_{\bf C} |f(z)|^2 e^{-|z|^2}dA\}$$ is a Bergman space of real analytic functions which contains non-holomorphic functions. 

The first step is not necessary for the coefficient estimates but seems independently interesting. 
We prove in Proposition 3 a kind of weak maximum  principle. 
\begin{prop}
Let $L\subseteq\bf C$ be a regular $n$-gon centered at $0$, $n\ge 5$,  with vertices $p_1,\dots,p_n$. 
Denote by $Y$ the set of vertices of $L$.  There exists $\delta>0$, and $0<r_1<r_2$, $C>0$  such that if $f\in {\cal A}_Y$, and $|z|<\delta$, then 
$$|f(z)|\le C\left(\int_{r_1<|z|<r_1}|f(z)|^pdA(z)\right)^{\frac{1}{p}}.$$
\end{prop}
\begin{proof}
The technique of the main lemma gives the estimate. The lower bound, $r_1$ results from choosing an analytic disc (affine) whose boundary avoids the leaves of foliations of the $M_i$'s containing 0, over the part of the arc where an estimate depends on a particular $M_i$. 
Let $M_1,\dots, M_n$ be the associated $CR$ half manifolds for the one dimensional extensions. In a small neighborhood of $z=0$, let $\widetilde{M_i}$ be the CR manifolds whose fiber over z is the same as the fiber of $M_i$ over 0. Each $\widetilde{M_i}$ is also a Levi flat. If construct an analytic disc with the desired properties with respect to the $\widetilde{M_i}$'s, then the same relation will hold between the analytic disc and the $M_i$'s, in a small enough neighborhood of 0.

The fiber over $z$ of  $\widetilde{M_i}$ is a ray from 0 in the direction $[0,\overline{p_i}]$. The angle between two adjacent  rays is $\frac{2\pi}{n}.$ 

Consider the circle $S$,  $z=0, |w|=\delta$ for a  small $\delta$ to be fixed later. For each point $q$ on $S$, there is an $L^p$ estimate for 
the value of $F$, expressed in terms of an integral in the $z$-plane. 
This estimate depends on the two closest $\widetilde M_i$'s, unless $q$  lies in a small neighborhood of $\widetilde M_i$, (the size of the neighborhood determines the coefficient in the $L^p$ estimate).
In that case, the adjacent $\widetilde M_i$'s are used. Because $n\ge 5$, the region of analytic continuation of those $\widetilde M_i$'s will cover $q$. Each $\widetilde M_i$ is determined by $\Im \phi_i=0$, where $\phi$ is a linear form. In general, the forms are quadratic; the linear approximation will suffice to understand the geometry. The holomorphic leaf on $M_i$ is determined by $Re\phi_i$. Therefore, if we show that $Re\phi_i(q)\ne 0, Re\phi_j(q)\ne 0$, where $\widetilde M_i$ and $\widetilde M_j$ are the Levi flats used for the estimate at $q$, then using the disc from Proposition 1, we get an integral over a region which does not contain $z=0$. By continuity we can find a non-zero lower bound.  

Write $q=(0, \delta e^{i\alpha})$. Fix an angle $\theta_0$; we will see that $\theta_0 =\frac{\pi}{18}$ will suffice.  We have two cases. 
\begin{enumerate}
\item $|\alpha-\frac{2l\pi}{n}|>\theta_0, l=1,2,\dots,n$.
\item $|\alpha-\frac{2l\pi}{n}|<\theta_0$, for some $l$. 
\end{enumerate}

In the first case you use the two closest $M_i$'s. In the second, you are close to one $M_i$ and use the adjacent ones on either side for estimation. It is evident that if $n\ge 5$ and $\theta <\frac{\pi}{10}$ for the pentagon (the worst case) that you have $Re(w_iq)>c\delta$ for some small $c>0$. 




An estimate for the original $M_i$'s and for all $z$ close enough to $0$, follows by continuity.

\end{proof}

By bounded point evaluation for $B_X$ and by scaling, we have that for any $p, |p|>1$, $|f(z))\le C\int_{|z|\le K|z}|f|^2dA(z)$,
where $C$ and $K$ are independent of $p$. This integral does not contain the Gaussian weight, since it depends on only on the geometry of the $M_a$'s. 
 We can derive a restriction on the growth of a function in $B_X$. As with analytic functions, set $M_r=\sup_{|z|=r}|f(z)|$.

\begin{theorem}
Let $f\in B_X$. The for any $t>\frac{K^2}{2}$, 
$$\limsup_{r\rightarrow \infty}\frac{M_r}{e^{tr^2}}=0.$$

\end{theorem}
\begin{proof}
Suppose not. Then for some $\delta>0$, there is a sequence $z_n$, $|z_n|\rightarrow \infty$ and $|f(z_n)|\ge \delta e^{t|z_n|^2}$.
We gave 
$$\delta e^{t|z|^2}\le C\int_{|z|\le K|z_n|}|f(z)|^2 dA(z)\le C e^{\frac{K^2}{2}|z_n|^2}\int_{|z|\le K|z_n|}|f(z)|^2e^{-|z|^2}dA(z).$$
From this we derive that $\frac{\delta}{C}e^{(t-\frac{K|z_n|^2}{2}}\le ||f||_2$, which forces the right side to be $\infty$. 
\end{proof}

Following the method of Levin, we can make progress in coefficient estimation. 
Let $f\in B_X$. Pick $M<\frac{K^2}{2}$. Then $|f(z)|<Ce^{M|z|^2}$.
For some $K_1$, $\sup_{|z|<r,|w|<r}<C\sup_{|z|<K_1r}|f(z)|$. Here, $f(z,w)$ is the complexified power series for $f$. 
The constant $K_1$ is not strictly related to the bounded point evaluation---we only need the Fr\`{e}chet space theorem here.

Putting these estimates together, we can say that for $t>\frac{K^2}{2}$ we have estimates of the form
$$sup_{|z|<r,|w|<r}|f(z,w)\le Ce^{tM^2r^2}.$$

Writing $f(z,w)=\Sigma_{n,m\ge 0}a_{mn}z^n(w g(z))^m)$ we can obtain some Cauchy estimates for coefficients. One would like to apply the Levin method, but the zeroes of $g(z)$ make it difficult to get a clean answer. In future work, we hope to show give conditions on the coefficient sequence of a function in $B_X$, and in reverse, to show how to estimate the norm based on coefficient estimates.

We also offer the following observation.
Suppose that the weight $\omega=e^{-k|z|}$ is used.
Then, depending on $k$, the growth of $\sin(z)$ suggests that 
$L^p({\bf C}, \omega)$ consists of series  which are polynomial in $\overline z$. In order to prove this, one needs coefficient estimation.
This example appears to relate to the topic of polyanalytic functions. What we would show is that with a weight $e^{-t|z|}$ we can construct Bergman spaces of polyanalytic functions: all of them real analytic, and with bounded point evaluation.

\section{Hybrid Bergman-Hardy spaces on CR manifolds}
Let $M\subseteq {\bf C^2}$ be a smooth $CR$ manifold such that the fiber over each point $z$ is a simple closed curve. Denote by $\Omega$ the domain which is the union of the interiors of the $M_z$'s. Assume that $\{(z,w):w=0\}\subseteq \Omega$ 
Given a discrete $X$ which satisfies the conditions of Theorem --, let $A_{M,X}=\overline{\pi^*({\cal A}_X)\otimes{\cal O}({\bf C})}$, where the closure is taken in the topology of uniform convergence on compacta. From \cite{La2}, an $f\in A_{M,X}$ is a continuous function on $M$  with  the following properties.
\begin{enumerate}
\item $\int_{M_z}f(z,w)dw \in {\cal A}_X$, where $M_z$ is the fiber of $M$ over $z$. 
\item every $f\in A_{M,X}$ extends continuously to a function $f(z,w)$ on $\Omega$ which is holomorphic in $w$ and real analytic in both variables. 
\end{enumerate}
For the next construction, we specialize to the case where the fibers of $M$ are circles centered at $0$. 
In this case $\Omega=\{(z,w)|w|\le e^{-\phi(z)}\}$, where $\phi$ is plurisubharmonic on $\bf C$. We assume that $\phi$ is not constant, which means that the fibers $M_z$ shrink to 0 as $|z|\rightarrow \infty$. 
Define a Bergman-Hardy space on $M$ as follows. We work in $L^p(M)$ with $\int_M g(z,w)d\mu=\int_{\bf C}\int_{M_z}g(z,w)|dw|dA(z)$. 

\begin{defn}
$B_{M,X,p}$ is the closure in $L^p(M,e^{-|z|^2}d\mu))$ of 
$\{f\in A_{M,X}: \int_M |f(z,w)|^pe^{-|z|^2}d\mu<\infty\}$. 
\end{defn}
Combining information about $A_{M,X}$ with basic Hardy space theory, we have the following proposition.
\begin{prop}
If $f\in B_{M,X,p}$, then the following hold.
\begin{enumerate}
\item $f$ extends to a real analytic function on $\Omega$ which is holomorphic in $w$, whose non-tangential  limits  in the  vertical ($w$) direction equal $f$ almost everywhere.
\item for almost every $z$, $f|{M_z}\in H^p{M_z}$, 
\item as $t\rightarrow 1$, $\int_{tM}|f(z,w)|^pe^{-|z|^2} td\mu\rightarrow \int_M|f(z,w)|^pe^{-|z|^2} td\mu$, where $tM$ is the dilation of $M$ in the vertical fiber, and $td\mu$ is the measure which is scaled by $t$ on each fiber in the obvious way.
\end{enumerate}
\end{prop}
The third point follows from the monotone convergence theorem; the second follows from standard theory, and the first was proved in \cite{La2}. 
Any $f\in B_{M,X,p}$ has  a series representation
$f(z,w)=\Sigma_{l,m,n}a_{lmn}z^m(\overline z g(z))^nw^l$. To find estimates for the coefficients,
take $b_{m}(z)=\frac{1}{2\pi i}\int_{M_z}\frac{f(z,w)}{w^n}dw.$
The $b_{n}$'s may not be integrable in the given weight, but  are in $L^p\{{\bf C}, e^{-|z|^2-|z|}\}$. This means that coefficient estimates, similar to those for $B_{X,\omega}$, can be obtained.

\section{A Bergman space on the unit disc}
It is clear from the basic construction of algebras that a sufficiently dense discrete  $X$ set of zeros in the unit disc, results in a  closed Fr\'echet algebra of real analytic functions; with more restrictions on the zero set of $g$, local $L^p$ estimates for bounded point evaluation will exist. 
The interesting case looks like $p=2$ and the function $g(z)$ is bounded analytic. This of course implies a restriction on the zero set of $g$.
It's not clear whether such an example could exist; however, let us assume such a situation and look at associated shift or shift-type operators.

So we suppose that $g(z)$ is a bounded analytic function on the unit disc with zeroes at $X\subseteq \Delta$, such that 
there is $L^2$ bounded point evaluation for ${\cal A}_{X,2}$, which is what we will call this space of functions.
An important distinction here is that a function $f\in {\cal A}_{X,2}$ need not be represented by a power series which converges on the entire disc. What we can do is say that $f(z)=\Sigma_{m,n\ge 0}a_{mn}z^m(\overline{z}g(z))^n$ in neighborhood of 0; furthermore, this representation is unique. Looking at the power series, we can construct shift operators. 
\begin{enumerate}
\item The classical shift operator $f\rightarrow zf$. 
\item Multiplication by $(\overline z)(g(z))$. It is in this step that we would use the fact that $g$ is bounded.
\item One can multiply the shifts from 1) and 2) together. You could also combine with projection onto the orthogonal complement in 
${\cal}_{X,2}$ of the usual Bergman space. 
\end{enumerate}

This suggestion is for mathematicians who know more about shift operators than the author. We cannot say whether any of these shift operators are interesting. 

\begin{defn} Let $X$ be a discrete set a domain $\Omega$. Denote by ${\cal A}_{\Omega, X}$ the set of continuous functions $f(z)$ on $\Omega$ such that if $p \in X$ and 
$\{|z-p|\le r\}\subseteq \Omega$, then $f|_{|z-a|=r}$ extends holomorphically to $|z-a|<r$. 
\end{defn}

It is clear from the construction of real analytic Fr\'{e}chet spaces that for $X$ which is thick enough at $\partial \Omega$, the functions in ${\cal A}_{\Omega,X}$ will be real analytic. 
Let $g(z)$ be an analytic function on $\Omega $ with simple zeroes exactly at the points of $X$.  Then $z$ and $\overline z g(z)$ generate an algebra, but it is not clear that this algebra is
${\cal A}_{\Omega, X}$. With a suitable weight on $\Omega$ one can construct a Bergman space.  We describe a construction for the unit disc.

The proof of the following lemma is elementary.

\begin{lemma}Let  $p1, p2, . . . , p5$ be the vertices of a convex pentagon $Q$ in $\bf C$. Let
$L1, L2, . . . , L5$ be the secants of $Q$, and let $s1, . . . , s5$ be the intersection points
of the secants which are in the interior of $Q$. They are the vertices of another
convex pentagon $\widetilde Q$. Then the geometric conditions for applying Theorem 1 to
${\cal A}_{p_1...p_5}$ hold for any point in the interior of $\widetilde Q$. In the case of the algebra ${\cal A}_{\Omega,X}$
the condition applies, if for each point $w \in Q$, and each $p_i$, the disc centered
at $p_i$ of radius $|p_i-w|$ is contained in $\Omega$.
\end{lemma}

Using this lemma, we can estimate how dense $X$ has to be at the circle to obtain a real analytic Bergman space. 

\begin{theorem}
There is a discrete set $X\subseteq D, X = \{a1, a2, . . .\}$ such that there
is local bounded point evaluation in $L_p$
for every point in the disc.
\begin{enumerate}
\item For any $ p > 2$,  there is a local $L_p$
bounded point evaluation estimate.
\item  For any $t > 1, \Sigma_{n=0}^{\infty}(1-|a_n|)^t <\infty$.
\end{enumerate}
\end{theorem}

\begin{proof}
Using Lemma 3, we can cover the annulus $1-2^{-n}\le z<1-2^{-n-1}$ 
the union of $k2^n$
 pentagons ($k$ independent of $n$) whose rectangular part has
side length $O(2^{-n-2})$
 guaranteeing that the estimate of Lemma 3 will hold at
interior points. Then the estimate on approach of zeros follows immediately

\end{proof}
If one could get a zero set with the Blaschke condition, then using a bounded
$g(z)$, one could construct a Bergman space of real analytic functions in $L^p(D)$.
The preceding theorem indicates that for some weight $\omega$, there is a Bergman
space of real analytic functions contained in $L_p(D,\omega)$. 
We denote this space by $B_{D,X,p,\omega}$. 

It would be interested to find out if $B_{D,X,p,\omega}$ is generated by $z$ and $\overline z g(z)$.

\section{Application to partial differential equations and further questions}

On a formal level, one can construct many PDE's of evolution type with constant coefficients, both linear and non-linear, such that solutions with initial data in ${\cal A}_X$ remain in ${\cal A}_ X$  for positive $t$. Consider the equation 
$u_t+u_{\overline z}=0$, or $u_t+u_{\overline z \overline z}$ for an unknown function $u(t,z)$ with initial conditions $u(0,\cdot)$. 
If $u(t,z)$ has the form $\Sigma_{n,m\ge 0}u_{nm}tz^n(\overline z g(z)^n$, then any number of 
$\frac{\partial}{\partial \overline z}$ derivatives remain formally in ${\cal A}_X$, which means that if $u(0,\cdot)\in {\cal A}_X$, then $u(t,\cdot)\in{\cal A}_z$ for $t>0$, formally.
We take the modified transport equation $u_t+u_{\overline z}=0$ to demonstrate how this can work for initial data which are polynomial in $\overline z$.
Suppose $u(0,z)=\Sigma_{m=0}^k u_m(t,z)(\overline z g(z))^k$, and we want to find a solution of the transport equation with $u_m$ holomorphic in $z$. If the initial conditions are constants $u_k(0,z)=U_k, 0\le k\le m, U_k=0, k>m$, then it is simple to show that $u_l(t,z)=0$ for $l>m$ and $u_k(t,z)$ is a polynomial in $t$ of degree $m-k$, for $0\le k\le m$. This is  a simple example. Certainly stronger results could be proved, if there is some potential application.



  \section{Approximation of Fock spaces by real analytic Bergman spaces}
The results of this section could be stated in greater generality, but we focus on the well known example of Fock spaces to illustrate the principle. 
Let $X\subseteq \bf C$ be a discrete set having the property that $nX\subseteq X$. 
Suppose also that for every positive integer $n$,  $\frac{1}{n}X\subseteq X$ . With this assumption, the point evaluation estimates for ${\cal A}_{ X}$ also hold for all ${\cal A}_n={\cal A}_{\frac{1}{n}X}$
The following theorem is clear.
\begin{theorem} $\cap_n {\cal A}_n =B$, the Fock space. 
\end{theorem}

Let $T_n$ be the orthogonal projection onto ${\cal A}_n$. Because of bounded point evaluation, we can represent $T_n$ by a kernel function $$K_n(z,w): T_n(f)(z)=\int_{\bf C}K_n(z,w)f(w)e^{-|z|^2|}dA(z).$$Let $K(z,w) $denote the kernel for the usual Bergman projection onto  Fock space. 
Because the point evaluation estimates are continuous, we have that $K_n(z,w)$ is a locally bounded sequence of functions. Clearly
$\lim_{n\rightarrow \infty}K_n(z,w)=K(z,w)$. From this, and because normal families real analytic spaces ${\cal A}_*$ behave like usual normal families,  we deduce that $K_n(z,w) \rightarrow K(z,w)$ locally uniformly on compacta, and real analytically as well. 

The Fock space is used in quantum mechanics in the Bargmann-Segal formalism. A possible application of our results is to extend the Segal-Bargmann formalism to these real analytic Bergman spaces. Perhaps there is an asymptotic formula as $n\rightarrow \infty$ in the limit which could be useful.

 \bibliographystyle{plain}

\end{document}